\documentclass[oneside,english]{amsart}
\usepackage[T1]{fontenc}
\usepackage[latin9]{inputenc}
\usepackage{amsbsy}
\usepackage{amstext}
\usepackage{amsthm}
\usepackage{amssymb}

\makeatletter
\numberwithin{equation}{section}
\numberwithin{figure}{section}
\theoremstyle{plain}
\newtheorem{thm}{\protect\theoremname}
\theoremstyle{plain}
\newtheorem{lem}[thm]{\protect\lemmaname}
\theoremstyle{plain}
\newtheorem{cor}[thm]{\protect\corollaryname}
\theoremstyle{definition}
\newtheorem{defn}[thm]{\protect\definitionname}

\makeatother

\usepackage{babel}
\providecommand{\corollaryname}{Corollary}
\providecommand{\definitionname}{Definition}
\providecommand{\lemmaname}{Lemma}
\providecommand{\theoremname}{Theorem}

\begin{document}
\global\long\def\Z{\mathcal{Z}}%

\global\long\def\I{\mathcal{I}}%

\title{A Quaternionic Nullstellensatz}
\author{Gil Alon and Elad Paran}
\address{GA: Department of Mathematics and Computer Science, The Open University
of Israel, 1 University Road, Raanana, 4353701, Israel}
\email{gilal@openu.ac.il}
\address{EP: Department of Mathematics and Computer Science, The Open University
of Israel, 1 University Road, Raanana, 4353701, Israel}
\email{paran@openu.ac.il}
\begin{abstract}
We prove a Nullstellensatz for the ring of polynomial functions in
$n$ non-commuting variables over Hamilton's ring of real quaternions.
We also characterize the generalized polynomial identities in $n$
variables which hold over the quaternions, and more generally, over
any division algebra.
\end{abstract}

\maketitle

\section{Introduction}

\subsection{The quaternionic Nullstellensatz}

Let $F$ be a field, and consider the polynomial ring $R=F[x_{1},...,x_{n}]$.
For any ideal $I\subseteq R$ one attaches the zero locus $\mathcal{Z}(I)\subseteq F^{n}$,
the set of points where all the polynomials in $I$ vanish. For any
set $Z\subseteq F^{n}$ one attaches the ideal $\mathcal{I}(Z)$ consisting
of all the polynomials in $R$ which vanish on every point of $Z$.

In the case where $F$ is algebraically closed, Hilbert's \emph{Nullstellensatz}
states that for any ideal $I\subseteq R$ we have $\mathcal{I}(\mathcal{Z}(I))=\sqrt{I}$,
the radical of $I$. Consequently, the ideals which are of the form
$\mathcal{I}(Z)$ for some $Z\subseteq F^{n}$ are precisely the radical
ones, i.e. the ideals which are closed under taking roots.

As the Nullstellensatz is a cornerstone of algebraic geometry, it
is natural to look for extensions of it to other fields. 

In the case $F=\mathbb{R}$, the \emph{Real Nullstellensatz }\cite[Corollary 4.1.8]{RAG}
states that $\mathcal{I}(\mathcal{Z}(I))=\sqrt[\mathbb{R}]{I}$, where
$\sqrt[\mathbb{R}]{I}$ is the \emph{real radical} of $I$, defined
by

\[
\sqrt[\mathbb{R}]{I}=\left\{ f\in R:\exists m\geq1,\text{}l\geq0,\text{ }f_{1},...,f_{l}\in R,\text{ \ensuremath{f^{2m}+\sum f_{i}^{2}}\ensuremath{\in I}}\right\} 
\]
Moreover, the ideals of the form $\mathcal{I}(Z)$ for some $Z\subseteq\mathbb{R}^{n}$
are characterized as those satifying the following condition:
\[
\forall l\geq1,\,\,f_{1},...,f_{l}\in R,\,\,\sum f_{i}^{2}\in I\Rightarrow f_{1},...,f_{l}\in I
\]
(See \cite[Theorem 4.1.4]{RAG}). Such ideals are called \emph{Real
}ideals.

In the present work we prove analogous results over Hamilton's non-commutative
algebra of real quaternions, $\mathbb{H}$. Let $R=\mathbb{H}[x_{1},...,x_{n}]$
be the ring of \emph{quaternionic polynomial functions}, namely, the
functions $f:\mathbb{H}^{n}\rightarrow\mathbb{H}$ expressible in
the form
\begin{equation}
f(x_{1},...,x_{n})=\sum_{l=1}^{r}a_{l,1}x_{\mu_{l,1}}a_{l,2}x_{\mu_{l,2}}...a_{l,k_{l}}x_{\mu_{l,k_{l}}}a_{l,k_{l}+1}\label{eq:function form}
\end{equation}

Where $r\geq0$, $k_{l}\geq0$, $\mu_{l,p}\in\{1,..,n\}$ and $a_{l,p}\in\mathbb{H}$.
The ring operations are simply pointwise addition and multiplication
of functions. By definition, we can evaluate each element of $R$
at every point $a=(a_{1},...,a_{n})\in\mathbb{H}^{n}$. Moreover,
for each fixed $a\in\mathbb{H}$, evaluation at $a$ is a ring homomorphism
from $R$ to $\mathbb{H}$. Hence, the notions of $\mathcal{I}(Z)\subseteq R$
for $Z\subseteq\mathbb{H}^{n}$ and $\mathcal{Z}(I)\subseteq\mathbb{H}^{n}$
for an ideal $I\subseteq R$ can be defined as in the commutative
case, and $\mathcal{I}(Z)$ is an ideal of $R$.\footnote{All the ideals in this paper are two-sided.}

The conjugation map on $\mathbb{H}$ is defined (for any $a,b,c,d\in\mathbb{R})$
by
\[
\overline{a+b\boldsymbol{i}+c\boldsymbol{j}+d\boldsymbol{k}}=a-b\boldsymbol{i}-c\boldsymbol{j}-d\boldsymbol{k}
\]
 For any $f\in R$, let $\overline{f}$ be the function $\overline{f}(a)=\overline{f(a)}$.
As we show below, we always have $\overline{f}\in R$ as well. Our
main theorem states:
\begin{thm}
\label{thm:quat_null_intro}\emph{(1)} Let $I$ be an ideal of $R$.
Then $\mathcal{I}(\mathcal{Z}(I))=\sqrt[\mathbb{H}]{I}$, where $\sqrt[\mathbb{H}]{I}$
(called the quaternionic radical of $I$) is defined by
\[
\sqrt[\mathbb{H}]{I}=\left\{ f\in R:\exists m\geq0,\text{\,\,}l\geq0,\text{ \,}f_{1},...,f_{l}\in R,\text{ \ensuremath{\left(f\overline{f}\right)^{m}+\sum f_{i}\overline{f_{i}}}\ensuremath{\in I}}\right\} 
\]

\emph{(2)} The ideals of the form $\mathcal{I}(Z)$ for some $Z\subseteq\mathbb{H}^{n}$,
are precisely the ideals which satisfy the following condition:

\[
\forall l\geq1,\,\,f_{1},...,f_{l}\in R,\,\,\sum f_{i}\overline{f_{i}}\in I\Rightarrow f_{1},...,f_{l}\in I
\]

We call such ideals quaternionic ideals. 
\end{thm}

We also characterize the quaternionic radical of an ideal $I$ as
the intersection of all prime quaternionic ideals containing it, in
analogy with similar theorems over $\mathbb{C}$ and $\mathbb{R}$.
See Corollary \ref{cor:prime_intersection} below.

\subsection{Polynomial functions}

Our proof of Theorem \ref{thm:quat_null_intro} relies on some algebraic
properties of the ring of polynomial functions over a division algebra.

Let $D$ be a division algebra, $F=C(D)$ its center, and let $m=[D:F]$.
Let $R=D[x_{1},...,x_{n}]$ be the ring of polynomial functions in
$x_{1},...,x_{n}$, namely, functions $f:D^{n}\rightarrow D$ expressible
in the form (\ref{eq:function form}) with coefficients $a_{l,p}\in D$.
For any $k\geq1$, let $D_{c}[x_{1},...,x_{k}]$ be the ring of polynomials
in $k$ \emph{central }variables: These polynomials take the classical
form 
\[
\sum a_{I}x^{I}=\sum_{0\leq i_{1},...,i_{k}\leq N}a_{i_{1},...,i_{k}}\prod_{l=1}^{k}x_{l}^{i_{l}}
\]
 for some $N\geq0$ and $a_{I}\in D$. Their multiplication is defined
by 
\[
\left(\sum a_{I}x^{I}\right)\left(\sum b_{J}x^{J}\right)=\sum a_{I}b_{J}x^{I+J}
\]
We will denote the resulting ring by $D_{c}[x_{1},...,x_{k}]$ (the
subscript $c$ stands for the variables being central: Indeed, the
variables $x_{1},...,x_{k}$ are in the center of this ring). For
any $k$-tuple $a=(a_{1},...,a_{k})\in D^{k}$, there is an evaluation
map $e_{a}:D_{c}[x_{1},...,x_{k}]\rightarrow D,$ defined by $e_{a}(f)=f(a)$.
However, unlike the case of $D[x_{1},...,x_{k}]$, $e_{a}$ is a ring
homomorphism if and only if $a\in F^{k}$. We will only evaluate elements
of $D_{c}[x_{1},...,x_{k}]$ at such points.

It was proven by Wilczy\'{n}ski \cite[Theorem 4.1]{Wilczy=000144ski}
that there is, in fact, a ring isomorphism

\begin{equation}
D[x_{1},...,x_{n}]\cong D_{c}[y_{1},y_{2},...,y_{mn}]\label{eq:main isomorphism}
\end{equation}
 In our proof of Theorem \ref{thm:quat_null_intro} we use an explicit
form of the above isomorphism, for which we prove a connection between
substitution of points in $D^{n}$ in elements of $D[x_{1},...,x_{n}]$
and substitution of points in $F^{mn}$ in the corresponding elements
of $D_{c}[y_{1},y_{2},...,y_{mn}]$. See Theorem \ref{eq:main isomorphism}
below.

\subsection{Generalized polynomial identities}

The final aspect of this work concerns \emph{generalized polynomial
identities. }These have been introduced by Amitsur in \cite{Amitsur},
as a generalization of the much studied polynomial identities. We
consider the algebra of formal non-commutative polynomials $A=D_{F}\langle x_{1},...,x_{n}\rangle:=D*_{F}F\langle x_{1},...,x_{n}\rangle$
(this follows the notation in \cite{Cohn}). Recall that each element
of $A$ is of the form (\ref{eq:function form}) with coefficients
in $D$, and that the scalars in $F$ are in the center of $A$. We
have a canonical homomorphism $A\rightarrow D[x_{1},...,x_{n}]$.
The kernel of the above homomorphism is the set of generalized polynomial
identities in $n$ variables over $D$. Let us denote this ideal by
$\text{GPI}(D,n)$. In Theorem \ref{thm:Generalized Identities} we
prove that $\text{GPI}(D,n)$ is finitely generated as an ideal of
$A$, and describe explicitly a set of $3\binom{m}{2}+1$ generators
of this ideal.

\subsection{The case $D=\mathbb{H}$}

To further explain our main ideas, let us write the isomorphism $\mathbb{H}[x]\cong\mathbb{H}_{c}[y_{1},y_{2},y_{3},y_{4}]$
explicitly. Consider the following elements of $\mathbb{H}[x]$:

\begin{align*}
Y_{1} & =\frac{1}{4}(x-\boldsymbol{i}x\boldsymbol{i}-\boldsymbol{j}x\boldsymbol{j}-\boldsymbol{k}x\boldsymbol{k})\\
Y_{2} & =\frac{1}{4}(\boldsymbol{j}x\boldsymbol{k}-x\boldsymbol{i}-\boldsymbol{i}x-\boldsymbol{k}x\boldsymbol{j})\\
Y_{3} & =\frac{1}{4}(\boldsymbol{k}x\boldsymbol{i}-x\boldsymbol{j}-\boldsymbol{j}x-\boldsymbol{i}x\boldsymbol{k})\\
Y_{4} & =\frac{1}{4}(\boldsymbol{i}x\boldsymbol{j}-x\boldsymbol{k}-\boldsymbol{k}x-\boldsymbol{j}x\boldsymbol{i})
\end{align*}
It is a matter of direct verification to see that for any quaternion
$q=a_{1}+a_{2}\boldsymbol{i}+a_{3}\boldsymbol{j}+a_{4}\boldsymbol{k}$
(with $a_{1},a_{2},a_{3},a_{4}\in\mathbb{R}$) we have $Y_{i}(q)=a_{i}$
for all $i$. So each $Y_{i}$, as a polynomial function, takes only
real values, and in particiular we have $Y_{i}Y_{j}=Y_{j}Y_{i}$ for
all $i,j$. Moreover, we have the identity $x=Y_{1}+\boldsymbol{i}Y_{2}+\boldsymbol{j}Y_{3}+\boldsymbol{k}Y_{4}$.
Define a homomorphism $\phi:\mathbb{H}[x]\rightarrow\mathbb{H}_{c}[y_{1},..,y_{4}]$
by $\phi(x)=y_{1}+\boldsymbol{i}y_{2}+\boldsymbol{j}y_{3}+\boldsymbol{k}y_{4}$.
Then $\phi$ is an isomorphism, as an inverse homomorphism can be
defined by $y_{i}\mapsto Y_{i}$.

As we show below, the ideal of generalized identities in one variable
over $\mathbb{H}$, $\text{GPI}(\mathbb{H},1)$ is generated by the
following 19 elements: $Y_{s}Y_{t}-Y_{t}Y_{s}$ (for $1\leq s<t\leq4$),
$Y_{s}a-aY_{s}$ (for $1\leq s\leq4$ and $a\in\{\boldsymbol{i},\boldsymbol{j},\boldsymbol{k}\}$)
and $x-\left(Y_{1}+\boldsymbol{i}Y_{2}+\boldsymbol{j}Y_{3}+\boldsymbol{k}Y_{4}\right)$.

We further note that substituting a quaternion $q=a_{1}+a_{2}\boldsymbol{i}+a_{3}\boldsymbol{j}+a_{4}\boldsymbol{k}$
(with $a_{1},a_{2},a_{3},a_{4}\in\mathbb{R}$) in any polynomial $f\in\mathbb{H}[x]$
has the same effect as substituting the real numbers $a_{1},a_{2},a_{3},a_{4}$
in $\phi(f)$. We can therefore establish a dictionary between systems
of equations in quaternionic variables and such systems in real variables.
Via this correspondence, the Real Nullstellensatz translates to a
quaternionic one, which is Theorem \ref{Our Nullstellensatz}.

\subsection{Outline}

The rest of this paper is organized as follows: In section 2, we collect
standard results from the theory of central simple algebras to prove
that an analog of the functions $Y_{1},..,Y_{4}$ exists for any division
algebra. In section 3, we discuss the isomorphism (\ref{eq:main isomorphism})
and describe a set of generators for $\text{GPI}(D,n)$. In section
4 we turn our attention to the case $D=\mathbb{H}$ and prove the
Nullstellensatz for $\mathbb{H}[x_{1},..,x_{n}]$.

\subsection*{Acknowledgement}

We are thankful to Bruno Deschamps for suggesting us to extend some
of the results in this paper, initially stated over the quaternions,
to any division algebra.

\section{Coordinate functions as noncommutative polynomials}

We start with the following lemmas, which use standard arguments on
central simple algebras.
\begin{lem}
\label{lem:indep}Let $D$ be a division algebra. If $a_{1},...,a_{n}\in D$
are linearly independent over the center $F=C(D)$, and $b_{1},..,b_{n}\in D$
are not all equal to $0$, then there exists $x\in D$ such that $\sum a_{i}xb_{i}\neq0$.
\end{lem}

\begin{proof}
For $n=1$ the claim is trivial. Assume the contrary of the claim,
and that $b_{1},...,b_{n}$ is a shortest counterexample, i.e. $n$
is minimal. Then all $b_{i}$ are nonzero. We may assume without loss
of generality that $b_{1}=1$. By our assumption, we have $\sum a_{i}xb_{i}=0$
for all $x\in D$. Let $b\in D$. By substituting $xb$ for $x$ in
the above equality we get $\sum a_{i}xbb_{i}=0$. Multiplying from
the right by $b$, we get $\sum a_{i}xb_{i}b=0$. Subtracting, we
get $\sum a_{i}x(bb_{i}-b_{i}b)=0$. Since $bb_{1}-b_{1}b=0$, we
must have by the minimality assumption $bb_{i}-b_{i}b=0$ for all
$i$. Hence $b_{i}\in F$ for all $i$. Setting $x=1$ we get $\sum a_{i}b_{i}=0$,
contrary to the independence assumption.
\end{proof}
\begin{lem}
\label{any endomorphism}Let $D$ be a division algebra, and let $F=C(D)$.
Then any $F$-linear endomorphism of $D$ can be expressed in the
following form: $f(x)=\sum_{i=1}^{r}a_{i}xb_{i}$ for some fixed $a_{i},b_{i}\in D$.
\end{lem}

\begin{proof}
Let $m=[D:F].$ Let $U$ be the $F$-vector space of maps $f:D\rightarrow D$
of the form $f(x)=\sum_{i=1}^{r}a_{i}xb_{i}$. We have $U\subseteq\text{End}_{F}(D)$.
Clearly, $\dim_{F}\text{End}_{F}(D)=m^{2}$. Let $v_{1},...,v_{m}$
be an $F$-linear basis of $D$. For any $1\leq i,j\leq m$ let $f_{ij}\in U$
be the function $f_{ij}(x)=v_{i}xv_{j}$. By Lemma \ref{lem:indep},
the functions $f_{ij}$ are $F$-linearly independent. As there are
$m^{2}$ such functions, we have $U=\text{End}_{F}(D)$.
\end{proof}
\begin{cor}
\label{cor:Existence of constants}Let $D$ be a division algebra
of dimension $m$ over its center $F=C(D)$. Let $v_{1},..,v_{m}$
be an $F$-linear basis of $D$. Then there exist $b_{st}^{i}\in F$
(for $1\leq i,s,t\leq m)$ such that for any $x=\sum_{i=1}^{m}c_{i}v_{i}$
(with $c_{i}\in F$), we have
\[
c_{i}=\sum_{s=1}^{m}\sum_{t=1}^{m}b_{st}^{i}v_{s}xv_{t}
\]
Moreover, the elements $b_{st}^{i}$ are unique.
\end{cor}

\begin{proof}
Let us define, for each $1\leq i\leq m$, $\phi_{i}:D\rightarrow F$
by $\phi_{i}(\sum_{i=1}^{m}c_{j}v_{j})=c_{i}$. Then $\phi_{i}$ is
$F$-linear, so by Lemma \ref{any endomorphism} it has a representastion
of the form $\phi_{i}(x)=\sum a_{j}xb_{j}$. The existence of $b_{st}^{i}$
follows by expressing each of the elements $a_{j},b_{j}$ as an $F$-linear
combination of $v_{1},...,v_{m}$. The uniqueness follows from Lemma
\ref{lem:indep}.
\end{proof}

\section{The ring isomorphism\label{sec:The-ring-isomorphism}}

We are now ready to prove our version of the isomorphism (\ref{eq:main isomorphism}):
\begin{thm}
\label{thm:main isomorphism}Let $D$ be a division algebra of dimension
$m>1$ over its center $F=C(D)$. Let $v_{1},...,v_{m}$ be an $F$-linear
basis of $D$. Let $b_{st}^{i}\in F$ be the elements deternmined
by $v_{1},...,v_{m}$ as in Corollary \ref{cor:Existence of constants}.
Consider the following elements $Y_{ij}\in D[x_{1},...,x_{n}]$ (for
$1\leq i\leq n$, $1\leq j\leq m$):
\[
Y_{ij}=\sum b_{st}^{j}v_{s}x_{i}v_{t}
\]
Then there exists a unique isomorphism 
\[
\phi:D[x_{1},...,x_{n}]\xrightarrow{{\sim}}D_{c}[y_{ij}:1\leq i\leq n,1\leq j\leq m]
\]
satisfying $\phi(Y_{ij})=y_{ij}$ for all $i,j$. Moreover, we have
for any $\left(a_{ij}\right)\in M_{n\times m}(F)$, and any $f\in D[x_{1},...,x_{n}]$,
\begin{equation}
f\left(\sum a_{1j}v_{j},...,\sum a_{nj}v_{j}\right)=\phi(f)\left(\left(a_{ij}\right)\right)\label{eq:subst formula}
\end{equation}
\end{thm}

\begin{proof}
By Corollary \ref{cor:Existence of constants}, any substitution of
$a_{1},...,a_{n}\in F$ in $Y_{ij}$ yields an element of $F$. Hence,
the elements $Y_{ij}$ commute with the elements of $D$ and with
each other. Moreover, by the same corollary, for any $1\leq i\leq n$
the following identity holds: 
\begin{equation}
x_{i}=\sum_{j}Y_{ij}v_{j}\label{identity for xi}
\end{equation}
Let us now define a homomorphism $\psi:D_{c}[y_{ij}]\rightarrow D[x_{1},...,x_{n}]$
by $\psi(y_{ij})=Y_{ij}$ for all $i,j$, and generally,
\[
\psi(\sum a_{I}(y_{ij})^{I})=\sum a_{I}(Y_{ij})^{I}
\]

Since the elements $Y_{ij}$ commute with the elements of $D$ and
with each other, $\psi$ is indeed a homomorphism.

For any $\left(a_{ij}\right)\in M_{n\times m}(F),$ let $a_{i}=\sum a_{ij}v_{j}$.
By Corollary \ref{cor:Existence of constants}, we have 
\[
Y_{ij}(a_{1},...,a_{n})=a_{ij}
\]
Therefore, for any $g=\sum a_{I}(y_{ij})^{I}\in D_{c}[y_{ij}]$ we
have (since substitution of $a_{1},...,a_{n}$ is a homomorphism from
$D[x_{1},...,x_{n}]$ to $D$)
\begin{align*}
g\left(a_{ij}\right) & =\sum a_{I}(a_{ij})^{I}\\
 & =\sum a_{I}(Y_{ij}(a_{1},...,a_{n}))^{I}\\
 & =\left(\sum a_{I}(Y_{ij})^{I}\right)(a_{1},...,a_{n})\\
 & =\psi(g)(a_{1},...,a_{n})
\end{align*}
This implies that $\psi$ is injective: Indeed, for $0\neq g\in D_{c}[y_{ij}]$
there exists a substitution $(a_{ij})\in M_{n\times m}(F)$ such that
$g(a_{ij})\neq0$ (note that by Wedderburn's theorem, since $m>1$,
$F$ must be infinite). By the above identity $\psi(g)$ attains a
nonzero value, so $\psi(g)\neq0$. $\psi$ is also surjective since
by (\ref{identity for xi}), $x_{i}=\psi(\sum v_{j}y_{ij})$. Hence
$\psi$ is an isomorphism. Let $\phi$ be the inverse of $\psi$,
then $\phi$ satisfies all the desired properties.
\end{proof}
As a consequence, we get a set of generators for the generalized polynomial
identities over $D$.
\begin{thm}
\label{thm:Generalized Identities}Under the same notation as in Theorem
\ref{thm:main isomorphism}, the ideal of identities $GPI(D,n)=\text{ker}(D\langle x_{1},...,x_{n}\rangle\rightarrow D[x_{1},...,x_{n}])$
is generated by the following elements:
\begin{itemize}
\item $v_{k}Y_{ij}-Y_{ij}v_{k}$ for $1\leq i\leq n$, $1\leq j,k\leq m$
\item $Y_{ij}Y_{i'j'}-Y_{i'j'}Y_{ij}$ for $1\leq i,i'\leq n$, $1\leq j,j'\leq m$
\item $x_{i}-\sum_{j}Y_{ij}v_{j}$ for $1\leq j\leq m$
\end{itemize}
\end{thm}

\begin{proof}
Let $f:D\langle x_{1},...,x_{n}\rangle\rightarrow D[x_{1},...,x_{n}]$
be the canonical map, and let $K\subseteq D\langle x_{1},...,x_{n}\rangle$
be the ideal generated by the above elements. We shall prove that
$K=\text{ker}f$. We have seen in the proof of Theorem \ref{thm:main isomorphism}
that the elements $Y_{ij}$ have all their values (as functions from
$D^{n}$ to $D$) in $F$, and that the identities $x_{i}=\sum_{j}Y_{ij}v_{j}$
hold in $D$, so all the elements in the above list are indeed in
the kernel of $f$, i.e $K\subseteq\text{ker}f$. Now let $p\in\text{ker}f$.
Using the above relations we can tranform $p$ into an element of
the form $p'=\sum a_{I}(Y_{ij})^{I}$ where the sum is on multi-indices
$I$ which are multisets of pairs $i,j$, and $a_{I}\in D$. We therefore
have (since $p-p'\in K)$ 
\[
0=\phi(f(p))=\phi(f(p'))=\sum a_{I}\left(y_{ij}\right)^{I}
\]
Therefore $a_{I}=0$ for all $I$. We conclude that $p'=0$, hence
$p\in K$.
\end{proof}

\section{The Nullstellensatz}

We now focus on the case $D=\mathbb{H}$, and set $R=\mathbb{H}[x_{1},...,x_{n}]$.
As mentioned in the introduction, there is a conjugation map on $\mathbb{H}$,
defined by 
\[
\overline{a+b\boldsymbol{i}+c\boldsymbol{j}+d\boldsymbol{k}}=a-b\boldsymbol{i}-c\boldsymbol{j}-d\boldsymbol{k}
\]
 Accordingly, for any $f\in R$ the conjugate function $\overline{f}:\mathbb{H}^{n}\rightarrow\mathbb{H}$
is defined by $\overline{f}(a)=\overline{f(a)}$.
\begin{lem}
For any $f\in R$ we have $\overline{f}\in R$.
\end{lem}

\begin{proof}
It follows from the formula of $Y_{1}$ in the introduction that the
following holds for any $a\in\mathbb{H}$:

\[
\overline{a}=-\frac{1}{2}(a+\boldsymbol{i}a\boldsymbol{i}+\boldsymbol{j}a\boldsymbol{j}+\boldsymbol{k}a\boldsymbol{k})
\]

Hence, for any $f\in R$, we have an equality of functions
\[
\overline{f}=-\frac{1}{2}(f+\boldsymbol{i}f\boldsymbol{i}+\boldsymbol{j}f\boldsymbol{j}+\boldsymbol{k}f\boldsymbol{k})
\]
\end{proof}
For any ideal $I\subseteq R$, we define the zero locus of $I$ as
in the commutative case:

\[
\mathcal{Z}_{\mathbb{H}}(I)=\left\{ a=(a_{1},...,a_{n})\in\mathbb{H}^{n}:f(a)=0\text{ for all }f\in I\right\} 
\]
Similarly, for any set $A\subseteq\mathbb{H}^{n}$ we define the ideal
of $A$ by
\[
\mathcal{I}_{\mathbb{H}}(A)=\left\{ f\in R:f(a)=0\text{ for all \ensuremath{a\in A}}\right\} 
\]

\begin{defn}
$\text{ }$
\begin{enumerate}
\item We will call an ideal $I\subseteq R$ \emph{quaternionic }if the following
condition holds: For any $f_{1},...,f_{k}\in R$, if $\sum f_{i}\overline{f_{i}}\in I$
then $f_{i}\in I$ for all i.
\item For any ideal $I\subseteq R$, the \emph{quaternionic radical }of
$I$, $\sqrt[\mathbb{H}]{I}$, is
\[
\sqrt[\mathbb{H}]{I}=\left\{ f\in R:\exists m\geq1,\,\,k\geq0,\text{ }f_{1},...,f_{k}\in R,\text{ \ensuremath{\left(f\overline{f}\right)^{m}+\sum f_{i}\overline{f_{i}}}\ensuremath{\in I}}\right\} 
\]
\end{enumerate}
\end{defn}

\begin{thm}
(Quaternionic Nullstellensatz).\label{Our Nullstellensatz}$\text{ }$Let
$I\subseteq R$ be an ideal.
\begin{enumerate}
\item $I=\mathcal{I}_{\mathbb{H}}(\mathcal{Z}_{\mathbb{H}}(I))$ if and
only if $I$ is quaternionic.
\item We always have $\mathcal{I}_{\mathbb{H}}(\mathcal{Z}_{\mathbb{H}}(I))=\sqrt[\mathbb{H}]{I}$.
\end{enumerate}
\end{thm}

\begin{proof}
(1) Consider the ring 
\[
S=\mathbb{H}_{c}[y_{ij}:1\leq i\leq n,1\leq j\leq4]
\]
and its subring

\[
S'=\mathbb{R}[y_{ij}:1\leq i\leq n,1\leq j\leq4]
\]

By Theorem \ref{thm:main isomorphism}, we have an isomorphism $\phi:R\xrightarrow{{\sim}}S$.
By Hilbert's basis theorem, $S$ is noetherian, and hence, so is $R$.
By \cite[Proposition 17.5]{GoodReal}, $S'$ is the center of $S$,
any ideal in $S$ is generated by elements in $S'$, and there is
a one-to-one correspondence between ideals in $S$ and ideals in $S'$
given by the following operations: An ideal in $E\subseteq S$ corresponds
to its intersection with $S'$, and an ideal $E'\subseteq S'$ corresponds
to its extension of scalars, 
\[
E'\otimes\mathbb{H}=\left\{ a+b\boldsymbol{i}+c\boldsymbol{j+}d\boldsymbol{k}:a,b,c,d\in E'\right\} \subseteq S
\]

Combining the isomorphism $\phi$ with the above correspondence, we
obtain a one-to-one corresponcence between the ideals of $R$ and
those of $S'$. Let $I'=\phi(I)\cap S'$, the ideal corresponding
to $I$.

Consider the bijection $\rho:\mathbb{H}^{n}\xrightarrow{{\sim}}\mathbb{\mathbb{R}}^{4n}$
defined by 
\[
\rho\left((a_{s1}+a_{s2}\boldsymbol{i}+a_{s3}\boldsymbol{j}+a_{s4}\boldsymbol{k})_{1\leq s\leq n}\right)=(a_{st})_{1\leq s\leq n,1\leq t\leq4}
\]

By the substitution formula (\ref{eq:subst formula}), we have for
any $f\in R$ and $a\in\mathbb{H}^{n}$, 
\begin{equation}
f(a)=\phi(f)(\rho(a))\label{eq:2ndsubstformula}
\end{equation}

For an ideal $I$ in $S$ or $S'$, let us denote the zero locus of
$I$ in $\mathbb{R}^{4n}$ by $\mathcal{Z}_{\mathbb{R}}(I)$.

By (\ref{eq:2ndsubstformula}) there is a one-to-one correspondence
between $\Z_{\mathbb{H}}(I)$ and the real zero locus of $\phi(I)$:
\begin{equation}
\rho(\Z_{\mathbb{H}}(I))=\Z_{\mathbb{R}}(\phi(I))=\Z_{\mathbb{R}}(I')\label{eq:rho_z}
\end{equation}
The last equality follows from $I=I'\otimes\mathbb{H}$.

Let $\psi$ be (as previously denoted) the inverse of $\phi$. For
any $B\subseteq\mathbb{R}^{4n}$ let $\I_{S}(B)$ and $\I_{S'}(B)$
be the set of elements of $S$ (resp. $S'$) which vanish on all the
points of $B$. By (\ref{eq:2ndsubstformula}) we have, for any $A\subseteq\mathbb{H}^{n}$:
\begin{align}
\I_{\mathbb{H}}(A) & =\psi(\I_{S}(\rho(A)))\nonumber \\
 & =\psi(\I_{S'}(\rho(A))\otimes\mathbb{H})\label{eq:IHA}\\
 & =\psi(\I_{S'}(\rho(A)))\otimes\mathbb{H}\nonumber 
\end{align}

We now apply results from real algebraic geometry. Recall the definitions
of a real ideal and the real radical, and the statement of the real
Nullstellensatz, all mentioned in the introduction. We shall use the
notation $\Z_{\mathbb{R}}(J)$ for the real locus of an ideal $J$
in $S'$.

Let us return to the ideal $I\subseteq R$. If $I=\I_{\mathbb{H}}(\Z_{\mathbb{H}}(I))$
then by the definition of $\I_{\mathbb{H}}$, $I$ is quaternionic.
On the other hand, if $I$ is quaternionic then $I'$  is a real ideal
of $S'$: Indeed, if $f_{i}\in S'$ and $\sum f_{i}^{2}\in I'$ then
$\sum\psi(f_{i})^{2}\in I$. Since $\psi(y_{ij})=Y_{ij}$, $\psi(f_{i})$
is a polynomial with real coefficients in the elements $Y_{ij}$,
so $\psi(f_{i})$ attains only real values. Hence, $\psi(f_{i})=\overline{\psi(f_{i})}$.
We get $\sum\psi(f_{i})\overline{\psi(f_{i})}\in I$, hence $\psi(f_{i})\in I$
for all $I$, hence $f_{i}\in I'$ for all $i$. By the Real Nullstellensatz,
we conclude that $\I_{S'}(\Z_{\mathbb{R}}(I'))=I'$. Hence, by (\ref{eq:IHA})
and (\ref{eq:rho_z}),

\begin{align*}
\I_{\mathbb{H}}(\Z_{\mathbb{H}}(I)) & =\psi(\I_{S'}(\rho(\Z_{\mathbb{H}}(I))))\otimes\mathbb{H}\\
 & =\psi(\I_{S'}(\Z_{\mathbb{R}}(I'))\otimes\mathbb{H}\\
 & =\psi(I')\otimes\mathbb{H}=I
\end{align*}
as desired. 

(2) Let $I$ be any ideal of $R$. If $f,f_{i}\in R$, $\ensuremath{\left(f\overline{f}\right)^{m}+\sum f_{i}\overline{f_{i}}}\in I$,
and $a\in\Z_{\mathbb{H}}(I)$, then we have $(f(a)\overline{f(a)})^{m}+\sum f_{i}(a)\overline{f_{i}(a)}=0$,
hence $f(a)=0$. This shows that $\sqrt[\mathbb{H}]{I}\subseteq\I_{\mathbb{H}}(\Z_{\mathbb{H}}(I))$.
On the other hand, using (\ref{eq:IHA}), (\ref{eq:rho_z}) and the
Real Nullstellensatz, we have
\begin{align*}
\I_{\mathbb{H}}(\Z_{\mathbb{H}}(I)) & =\psi(\I_{S'}(\rho(\Z_{\mathbb{H}}(I)))\otimes\mathbb{H}\\
 & =\psi(\I_{S'}(\Z_{\mathbb{R}}(I')))\otimes\mathbb{H}\\
 & =\psi\left(\sqrt[\mathbb{R}]{I'}\right)\otimes\mathbb{H}
\end{align*}
so any element of $\I_{\mathbb{H}}(\Z_{\mathbb{H}}(I))$ is of the
form
\[
f=f_{1}+f_{2}\boldsymbol{i}+f_{3}\boldsymbol{j}+f_{4}\boldsymbol{k}
\]
where $f_{i}\in\psi\left(\sqrt[\mathbb{R}]{I'}\right)$ for all $1\leq i\leq4$.
The elements $f_{i}$, being isomorphic images of elements of the
center of $S$, are in the center of $R$. By the definition of the
real radical, we can find some $m_{i}\geq1$ and $f_{ij}\in\psi(S')$
such that $f_{i}^{2m_{i}}+\sum_{j}f_{ij}^{2}\in\psi(I')$. Hence,
for any $y\in\psi(S')$ we have $y^{2}f_{i}^{2m_{i}}+\sum_{j}(yf_{ij})^{2}\in\psi(I')$.
Let us consider, for any $m\geq1$,
\[
\left(f\overline{f}\right)^{m}=\left(\sum_{i=1}^{4}f_{i}^{2}\right)^{m}
\]
for $m$ sufficiently large, when we expand the right-hand side by
the multinomial formula (keeping in mind that the summands commute),
we get that each summand is of the form $y^{2}f_{i}^{2m_{i}}$ for
some $y\in\psi(S')$ and $1\leq i\leq4$, hence can be complemented
by a sum of squares of elements in $\psi(S')$ to an element of $\psi(I')$.
We conclude that for some $m\geq1$ there are elements $g_{i}\in\psi(S')$,
such that 
\[
\left(f\overline{f}\right)^{m}+\sum g_{i}^{2}\in\psi(I')\subseteq I
\]
but $g_{i}=\overline{g_{i}}$, so $f\in\sqrt[\mathbb{H}]{I}$. Hence,
$\I_{\mathbb{H}}(\Z_{\mathbb{H}}(I))\subseteq\sqrt[\mathbb{H}]{I}$
and we have inclusions in both directions.
\end{proof}
\begin{cor}
The quaternionic radical of an ideal in $R$ is the minimal quaternionic
ideal containing it. \label{cor:The-quaternionic-radical}
\end{cor}

\begin{proof}
Indeed, let $I$ be an ideal of $R$. By Theorem \ref{Our Nullstellensatz},
$\sqrt[\mathbb{H}]{I}=\I_{\mathbb{H}}(\Z_{\mathbb{H}}(I))$ is a quaternionic
ideal. If $J$ is a quaternionic ideal containing $\I_{\mathbb{H}}(\Z_{\mathbb{H}}(I))$
then $\Z_{\mathbb{H}}(J)\subseteq\Z_{\mathbb{H}}(I)$, so $J=\I_{\mathbb{H}}(\Z_{\mathbb{H}}(J))\supseteq\I_{\mathbb{H}}(\Z_{\mathbb{H}}(I))$.
\end{proof}
The classical radical of an ideal $I$ in a commutative ring can be
characterized as the intersection of all prime ideals containing it.
Similarly, the real radical of an ideal $I$ in $\mathbb{R}[x_{1},....x_{n}]$
is the intersection of all real prime ideals containing it (see \cite[Proposition 4.1.7]{RAG}).
The following corollary gives an analogous characterization of the
quaternionic radical.
\begin{cor}
Let $I$ be an ideal of $R$. Then $\sqrt[\mathbb{H}]{I}$ is the
intersection of all prime quaternionic ideals containing $I$.\footnote{We follow Krull's definition of a prime ideal in a noncommutative
ring \cite{Krull}, which has the following equivalent phrasing: an
ideal $P$ in $R$ is prime if given $a,b\in R$ with $(a)(b)\subseteq P$
(where $(u)$ is the ideal generated by $u$), it follows that $(a)\subseteq P$
or $(b)\subseteq P$. } \label{cor:prime_intersection}
\end{cor}

\begin{proof}
By Corollary \ref{cor:The-quaternionic-radical}, $\sqrt[\mathbb{H}]{I}$
is the minimal quaternionic ideal containing $I$, so in particular
it is contained in all quaternionic prime ideals containing $I$.
Conversely, suppose that $a\in R$ is contained in all quaternionic
prime ideals containing $I$, but $a\notin\sqrt[\mathbb{H}]{I}$.
Let $J$ be maximal among all quaternionic ideals containing $I$
but not $a$. We shall show that $J$ is prime and thus get a contradiction.

Indeed, suppose there exist $b,b'\in R$ such that $\ensuremath{(b)\not\nsubseteq J}$
and $\ensuremath{(b')\nsubseteq J}$, but $\ensuremath{(b)(b')\subseteq J}$.
From the maximality assumption on $J$ it follows that $a\in\sqrt[\mathbb{H}]{J+(b)}$
and $a\in\sqrt[\mathbb{H}]{J+(b')}$. For any element $f\in R$, let
$N(f)\in\mathbb{R}[x_{1},\ldots,x_{n}]$ denote the norm $f\overline{f}$.
Then we can write $N(a)^{m}+\sum N(c_{i})\in J+(b)$ and $N(a)^{k}+\sum N(c'_{i})\in J+(b')$.
The product $(N(a)^{m}+\sum N(c_{i}))\cdot(N(a)^{k}+\sum N(c'_{i}))$
equals $N(a)^{m+k}$ plus a sum of norms, and this product belongs
to $(J+(b))\cdot(J+(b'))\subseteq J+J(b')+(b)J+(b)(b')\subseteq J$.
Thus $a\in\sqrt[\mathbb{H}]{J}=J$, a contradiction.
\end{proof}


\begin{thebibliography}{1}
\bibitem[1]{Amitsur}Amitsur, S. A. (1965). Generalized polynomial
identities and pivotal monomials. Transactions of the American Mathematical
Society, 114(1), 210-226.

\bibitem[2]{RAG}Bochnak, J., Coste, M., and Roy, M-F., \emph{Real
Algebraic Geometry}, Springer, 1998

\bibitem[3]{Cohn}Cohn, P. M. (1977). Skew field constructions (Vol.
27). CUP Archive.

\bibitem[4]{GoodReal}Goodearl, K.R. and Warfield, R.B., \emph{An
introduction to noncommutative Noetherian rings}, Cambridge, 1989.

\bibitem[5]{Krull}Krull, W., Primidealketten in allgemeinen Ringbereichen,
de Gruyter, 1928.

\bibitem[6]{LaurenceSimons}Lawrence, J. and Simons, G. E. (1989).
Equations in division rings---a survey. The American mathematical
monthly, 96(3), 220-232.

\bibitem[7]{Wilczy=000144ski}Wilczy\'{n}ski, D. M. (2014). On the
fundamental theorem of algebra for polynomial equations over real
composition algebras. Journal of Pure and Applied Algebra, 218(7),
1195-1205.
\end{thebibliography}
\end{document}